\documentclass[12pt]{article}

\usepackage{amsfonts,mathrsfs}
\usepackage{amssymb,amsmath,amsbsy,amsthm}
\usepackage{dsfont}
\usepackage{blkarray}
\usepackage{tikz}
\usepackage{tkz-euclide}
\usepackage{tikz-cd}
\usetikzlibrary{patterns}
\usepackage{ae}
\usepackage{bm}
\usepackage{graphicx}
\usepackage{float}
\usepackage{cite}
\usepackage{indentfirst}
\usepackage{lineno}
\usepackage{titlesec}
\usepackage{enumitem}
\usepackage{color}
\usepackage{aliascnt}
\usepackage{varioref}
\usepackage{hyperref}
\hypersetup{colorlinks=true,linkcolor=cyan,anchorcolor=cyan,citecolor=red,CJKbookmarks=True,
	,urlcolor=cyan}
\usepackage{cleveref}

\numberwithin{equation}{section}

\def\int{\mbox{\rm int}}

\def\min{\mbox{\rm min}}
\def\max{\mbox{\rm max}}

\def\And{\mbox{\rm ~and~}}

\def\For{\mbox{\rm ~for~}}

\def\black{\color{black}}

\def\({\mbox{\rm (}}\def\){\mbox{\rm )}}

\makeatletter

\newcommand{\Rmnum}[1]{\expandafter\@slowromancap\romannumeral #1@}
\makeatother
\newtheorem{theorem}{Theorem}[section]
\newaliascnt{lemma}{theorem}
\newtheorem{lemma}[lemma]{Lemma}
\aliascntresetthe{lemma}

\newaliascnt{proposition}{theorem}
\newtheorem{proposition}[proposition]{Proposition}
\aliascntresetthe{proposition}

\newaliascnt{fact}{theorem}

\aliascntresetthe{fact}

\newaliascnt{definition}{theorem}
\newtheorem{definition}[definition]{Definition}
\aliascntresetthe{definition}

\newaliascnt{conjecture}{theorem}

\aliascntresetthe{conjecture}

\newaliascnt{corollary}{theorem}
\newtheorem{corollary}[corollary]{Corollary}
\aliascntresetthe{corollary}

\newaliascnt{claim}{theorem}

\aliascntresetthe{claim}

\newaliascnt{problem}{theorem}

\aliascntresetthe{problem}

\newaliascnt{question}{theorem}

\aliascntresetthe{question}

\newaliascnt{remark}{theorem}

\aliascntresetthe{remark}

\newaliascnt{example}{theorem}
\newtheorem{example}[example]{Example}
\aliascntresetthe{example}

\newaliascnt{notation}{theorem}

\aliascntresetthe{notation}

\begin{document}
	
	\begin{center}
		{\Large\bf
			Level of Regions for Deformed Braid Arrangements
		}\\ [7pt]
	\end{center}
	\vskip 3mm
	\begin{center}
		Yanru Chen$^1$, Houshan Fu$^{2,*}$,  Suijie Wang$^3$ and Jinxing Yang$^{4}$\\[8pt]
		
		$^{1,3,4}$School of Mathematics\\
		Hunan University\\
		Changsha 410082, Hunan, P. R. China\\[12pt]
		
		$^2$School of Mathematics and Information Science\\
		Guangzhou University\\
		Guangzhou 510006, Guangdong, P. R. China\\[15pt]
		
		$^*$Correspondence to be sent to: fuhoushan@gzhu.edu.cn \\
		Emails: $^1$yanruchen@hnu.edu.cn, $^3$wangsuijie@hnu.edu.cn, $^{4}$jxyangmath@hnu.edu.cn\\[15pt]
	\end{center}
	\vskip 3mm
	
	\begin{abstract}
		This paper primarily investigates a specific type of deformation of the braid arrangement $\mathcal{B}_n$ in $\mathbb{R}^n$, denoted by $\mathcal{B}_n^A$  and defined in  \eqref{Deformation-Arrangement}. Let $r_l(\mathcal{B}_n^A)$ be the number of regions of level $l$ in $\mathcal{B}_n^A$ with the corresponding exponential generating function $R_l(A;x)$. Using the weighted digraph model introduced by Hetyei \cite{Hetyei2024}, we establish a bijection between regions of level $l$ in $\mathcal{B}_n^A$ and valid $m$-acyclic weighted digraphs on the vertex set $[n]$ with exactly $l$ strong components. Based on this bijection, we obtain a property analogous to a polynomial sequence of binomial type, that is, $R_l(A;x)$ satisfies the relation
		\[
		R_l(A;x)=\big(R_1(A;x)\big)^l=R_k(A;x)R_{l-k}(A;x).
		\]
		Furthermore, the values $r_l(\mathcal{B}_n^A)$ yield a combinatorial interpretation for the coefficients in the expansion of the characteristic polynomial $\chi_{\mathcal{B}_n^A}(t)$ in the basis elements $\binom{t}{l}$, that is,		\[\chi_{\mathcal{B}_n^A}(t)=\sum_{l=0}^n(-1)^{n-l}r_l(\mathcal{B}_n^A)\binom{t}{l}.\]
If $n$, $a$ and $b$ are non-negative integers with $n\ge 2$ and $b-a\ge n-1$, for the deformation $\mathcal{B}_n^{[-a,b]}$ defined in \eqref{Integer-Arrangement}, its characteristic polynomial has a single real root $0$ of multiplicity one when $n$ is odd, and has one more real root $\frac{n(a+b+1)}{2}$ of multiplicity one when $n$ is even.
		\vskip 6pt
		\noindent{\bf Keywords:} hyperplane arrangement, deformation of braid arrangement, exponential sequence of arrangements, sequence of binomial type, characteristic polynomial\vspace{1ex}\\
		{\bf Mathematics Subject Classifications:} 05C35, 05A10, 05A15, 11B83
	\end{abstract}
	
	\section{Introduction}\label{Sec-1}

 One of the most important hyperplane arrangements is the {\em braid arrangement} $\mathcal{B}_n$, which consists of the following hyperplanes in the $n$-dimensional Euclidean space $\mathbb{R}^n$,
	\[
	\mathcal{B}_n:\quad x_i-x_j=0,\quad 1\le i<j\le n.
	\]
	A {\em deformation} of $\mathcal{B}_n$ is obtained by replacing each hyperplane $x_i-x_j=0$ with hyperplanes:
	\begin{equation}\label{Deformation}
		x_i-x_j=a_{ij}^{(1)},a_{ij}^{(2)},\ldots,a_{ij}^{(n_{ij})}, \quad 1\le i<j\le n,
	\end{equation}
	where each $a_{ij}^{(k)}$ is a real number. Various special deformations of $\mathcal{B}_n$ have been extensively studied in \cite{Athanasiadis1998,P-S2000}. In this paper, we mainly study a specific type of deformation $\mathcal{B}_n^A$ of the braid arrangement $\mathcal{B}_n$ , which is defined as follows, for any finite set $A=\{a_1,\ldots,a_m\}$ of real numbers,
\begin{equation}\label{Deformation-Arrangement}
\mathcal{B}_n^A:\quad x_i-x_j=a_1,\ldots,a_m, \quad 1\le i<j\le n.
\end{equation}
Given integers $a,b$ with $a\le b$, let $[a,b]=\{a,a+1,\ldots,b\}$. When $b$ is a positive integer, we simply denote $[b]=[1,b]$ and $[-b]=[-b,-1]$. Then the deformation $\mathcal{B}_n^{[a,b]}$ of $\mathcal{B}_n$ is
\begin{equation}\label{Integer-Arrangement}
	\mathcal{B}_n^{[a,b]}: \quad x_i-x_j=a,a+1,\ldots,b,\quad1\le i<j\le n.
\end{equation}
	Several especial cases are noteworthy: $\mathcal{B}_n^{[1]}$ is the {\em Linial arrangement} and $\mathcal{B}_n^{[-b+2,b]}$ is the {\em extended Linial arrangement} with $b\ge 1$; $\mathcal{B}_n^{[0,1]}$ is the {\em Shi arrangement} and $\mathcal{B}_n^{[-b+1,b]}$ is the {\em extended Shi arrangement} with $b\ge 1$; $\mathcal{B}_n^{[-1,1]}$ is the {\em Catalan arrangement} and $\mathcal{B}_n^{[-b,b]}$ is the {\em extended Catalan arrangement} with $b\ge 1$; a more general case $\mathcal{B}_n^{[-a+1,b-1]}$ is the {\em truncated affine arrangement} studied in \cite{P-S2000} for non-negative integers $a+b\ge 2$.

Counting regions of a hyperplane arrangement is commonly done by computing its characteristic polynomial and applying it to Zaslavsky's formula \cite{Zaslavsky1975}. This approach triggers many significant combinatorial labelings of regions, such as the Pak-Stanley labeling \cite{D-G2021,Stanley1996,Stanley1998} for the regions of  extended Shi arrangements and extended Catalan arrangements, the Athanasiadis-Linusson labeling \cite{Athanasiadis1999} for the regions of the extended Shi arrangement, and the Hopkins-Perkinson labeling \cite{H-P2016} for counting regions of bigraphical arrangements. For more related works, please refer to \cite{F-W2021,Gessel2019,H-P2012,L-R-W2014,Tewari2019}.

This paper focuses on enumerative problems concerning regions of level $l$ in these deformations $\mathcal{B}_n^A$. A {\em region} of a hyperplane arrangement $\mathcal{A}$ in $\mathbb{R}^n$ is a connected component of its complement  $\mathbb{R}^n-\bigcup_{H \in \mathcal{A}} H$.   Ehrenborg \cite{Ehrenborg2019} introduced the {\em level} of a subset $X$ in $\mathbb{R}^n$, which is the smallest non-negative integer $l$ such that there exists a subspace $W$ of dimension $l$ and a positive real number $r$ satisfying
	\[
	X \subseteq\big\{\bm x \in \mathbb{R}^n : \min_{\bm w\in W} \|\bm x - \bm w\| \le r \big\},
	\]
denoted by $l(X)$. Particularly, when $X$ is a subspace or cone in $\mathbb{R}^n$, obviously $l(X)=\dim(X)$.
Ehrenborg \cite{Ehrenborg2019} provided an explicit expression on the number of faces of level $l$ and codimension $k$ for extended Shi arrangements. Armstrong and Rhoades \cite{Armstrong2012} independently defined the level of a convex set  relative to its recession cone. For any convex set $C\subseteq\mathbb{R}^n$, it determines a {\em recession cone} ${\rm Rec}(C):=\{\bm x\in\mathbb{R}^n: \bm x+C\subseteq C\}$. They defined the dimension of ${\rm Rec}(C)$ as the number of {\em degrees of freedom} of the convex set $C$, and revealed that these ``degrees of freedom'' are a fundamental aspect of the ``combinatorial symmetry'' between the Shi and Ish arrangements. It is clear that the level $l(C)$ equals the number of degrees of freedom of $C$. Zaslavsky \cite{Zaslavsky2003} defined the {ideal dimension} $\dim_{\infty}(X)$ of a real affine set $X$ as the dimension of the intersection of its projective topological closure with the infinite hyperplane. Furthermore,  Zaslavsky demonstrated that for any subset $X\subseteq\mathbb{R}^n$,  the relation $l(X) \geq \dim_{\infty}(X) + 1$ holds, with equality when $X$ is a polyhedron.  Based on their work, Leven, Rhoades, and Wilson \cite{L-R-W2014} constructed bijections that preserve levels between regions of deleted Shi and Ish arrangements. Most recently, Chen,Wang,Yang and Zhao \cite[Theorem 6.1, 6.2]{CWYZ2024} addressed several enumerative problems concerning the number of regions of level $l$ in Catalan-type and semiorder-type arrangements by constructing a labeled Dyck path model for their regions, and presented a formula for counting regions of level $l$ in extended Catalan arrangements. Hetyei \cite{Hetyei2024} developed a more general method for labeling regions of any deformation of the braid arrangement by weighted digraphs. In this work,  Hetyei showed that regions of any deformation of the braid arrangement can be bijectively labeled by a set of valid $m$-acyclic weighted digraphs, where relatively bounded regions correspond precisely to strongly connected valid $m$-acyclic weighted digraphs.

The present work is motivated by Chen,Wang,Yang and Zhao's paper \cite{CWYZ2024} and Hetyei's labeling \cite{Hetyei2024}. Note that both Catalan-type and semiorder-type arrangements are symmetric deformations of braid arrangements. One purpose of this paper is to extend their results to nonsymmetric deformations $\mathcal{B}_n^{A}$ of the braid arrangement, utilizing the weighted digraph model of regions introduced by Hetyei \cite{Hetyei2024}. Another focus is the investigation of real roots of the characteristic polynomial $\chi_{\mathcal{B}_n^{[-a,b]}}(t)$ of the hyperplane arrangement $\mathcal{B}_n^{[-a,b]}$ for any non-negative integers $a\le b$. In \autoref{Main-Lel}, we obtain that the level of any region in the hyperplane arrangement $\mathcal{B}_n^A$ equals the number of strong components of the associated valid $m$-acyclic weighted digraph. This enables us to refine Hetyei's work for $\mathcal{B}_n^A$. More specifically, we prove in \autoref{Bij-3} that the regions of level $l$ in $\mathcal{B}_n^A$ are in bijection with the valid $m$-acyclic weighted digraphs that have exactly $l$ strong components. Furthermore, by applying the weighted digraph model for counting regions, we demonstrate in \autoref{Main1} that the sequence $\big(r_l(\mathcal{B}_n^A)\big)_{n\ge 0}$ possesses a property analogous to that of a polynomial sequence of binomial type, which generalizes the results \cite[Theorem 1.2, 1.3, 1.4]{CWYZ2024}. Polynomial sequences of binomial type arising from various fields, including graph theory and partially ordered sets, were systematically summarized by Rota \cite{MR-1970,RKO-1973} as fundamental theory of combinatorics.

	
	Denote by $r(\mathcal{A})$ the total number of regions of a hyperplane arrangement $\mathcal{A}$ in $\mathbb{R}^n$ and $r_l(\mathcal{A})$ the number of regions of level $l$, with the convention that $r_0(\mathcal{A})=1$ when $n=0$. Obviously, $r_l(\mathcal{A})=0$ if  $l>n$.  Associated with $r_l(\mathcal{A})$, we define an exponential generating function as
	\begin{equation}\label{Def1}
		R_l(A;x):=\sum_{n\ge 0}r_l(\mathcal{B}_n^A)\frac{x^n}{n!}=\sum_{n\ge l}r_l(\mathcal{B}_n^A)\frac{x^n}{n!}.
	\end{equation}
	
	\begin{theorem}\label{Main1}
		Let $A=\{a_1,\ldots,a_m\}$ consist of distinct real numbers. For any non-negative integers $n$ and $l$, we have
		\[
		R_l(A;x)=\big(R_1(A;x)\big)^l.
		\]
		Equivalently, the relationship is given explicitly as follows:
		\begin{equation*}
			r_l\big(\mathcal{B}_n^A\big)=\sum_{\substack{n_1+\cdots+n_l=n,\\n_1,\ldots,n_l\ge 1}}\binom{n}{n_1,\ldots ,n_l}\prod_{i=1}^lr_1\big(\mathcal{B}_{n_i}^A\big).
		\end{equation*}
		Moreover, for any non-negative integer $k\le l$, we have
		\[
		R_l(A;x)=R_k(A;x)R_{l-k}(A;x),
		\]
		that is,
		\[
		r_l(\mathcal{B}_n^A)=\sum_{i=0}^n\binom{n}{i}r_{k}(\mathcal{B}_i^A)r_{l-k}(\mathcal{B}_{n-i}^A).
		\]
	\end{theorem}

	In 1996, Stanley \cite{Stanley1996} observed a fundamental yet important phenomenon that for an exponential sequence  $\mathfrak{A}=(\mathcal{A}_{1},\mathcal{A}_{2},\ldots)$ of arrangements, the characteristic polynomial of  $\mathcal{A}_n$ is actually determined by the number of regions, which is stated at the beginning of \autoref{Sec-3}. Recently, the authors in \cite{CWYZ2024} provided an explicit expression for the characteristic polynomial of the  Catalan-type and semiorder-type arrangement in terms of the number of regions of each level $l$. Notably, the sequence $\mathfrak{B}=(\mathcal{B}_1^A,\mathcal{B}_2^A,\ldots)$ is an exponential sequence of arrangements. Below we extend their work \cite[Theorem 1.5]{CWYZ2024} to a more general hyperplane arrangement $\mathcal{B}_n^A$.
	\begin{theorem}\label{Main2}
		Let $A=\{a_1,\ldots,a_m\}$ consist of distinct real numbers. For any non-negative integer $n$, the characteristic polynomial $\chi_{\mathcal{B}_n^A}(t)$  can be expressed in the following form
		\[
		\chi_{\mathcal{B}_n^A}(t)=\sum_{l=0}^n(-1)^{n-l}r_l(\mathcal{B}_n^A)\binom{t}{l}
		\]
		with the convention $\chi_{\mathcal{B}_0^A}(t)=1$.
	\end{theorem}
	
	It is worth remarking that Postnikov and  Stanley \cite{P-S2000} investigated roots of the characteristic polynomials of truncated affine arrangements and demonstrated that these arrangements satisfy the ``Riemann hypothesis". Motivated by their work, we further explore real roots of the characteristic polynomial $\chi_{\mathcal{B}_n^{[-a,b]}}(t)$ under the conditions that the non-negative integers $n\ge 2$, $a$ and $b$ satisfy $b-a\ge n-1$ in \autoref{Sec-5}.
	
	The paper is organized as follows. \autoref{Sec-2} introduces basic results related to the weighted digraph model and proves \autoref{Main1}. In \autoref{Sec-3}, we build on \autoref{Main1} to further verify \autoref{Main2}. \autoref{Sec-4} provides a formula for counting regions of level $l$ in $\mathcal{B}_n^{[-a,b]}$. Finally, we discuss real roots of the characteristic polynomial $\chi_{\mathcal{B}_n^{[-a,b]}}(t)$ in \autoref{Sec-5}.

\section{Proof of \autoref{Main1}}\label{Sec-2}
	
\subsection{Labeling regions of deformations of the braid arrangement}
	Let us first review some necessary definitions on digraphs. Let $D=\big(V(D),A(D)\big)$ be a finite digraph with the vertex set $V(D)$ and directed edge set $A(D)$. In the digraph $D$, two vertices $v$ and $u$ are said to be {\em strongly connected} if there exists a directed $(v,u)$-walk from $v$ to $u$ and a directed $(u,v)$-walk from $u$ to $v$. Alternatively,  two vertices $v$ and $u$ is strongly connected in $D$ if and only if there exists a directed cycle of $D$ containing $v$ and $u$. The digraph $D$ is called {\em strongly connected} if any two vertices $v$ and $u$ are strongly connected. Every maximal strongly connected subdigraph of $D$ is known as a {\em strong component} of $D$.
	
	Most recently, Hetyei \cite{Hetyei2024} constructed a weighted digraph associated with each region of any deformation of a graphical arrangement. Using this graph structure, Hetyei further demonstrated  that the regions  of such deformation can be bijectively labeled by a set of $m$-acyclic weighted digraphs, which contain only directed cycles with negative weights. Moreover, the relatively bounded regions correspond precisely to the strongly connected $m$-acyclic weighted digraphs. The weighted digraph model is a key approach for establishing our results.
	
	Describing a region in a deformation of the braid arrangement $\mathcal{B}_n$ amounts to determining whether a system of linear inequalities of the form
	\begin{equation}\label{Def-Weighted-Digraph}
		m_{ij}<x_i-x_j<M_{ij},\quad 1\le i<j\le n
	\end{equation}
	has a solution in $\mathbb{R}^n$. We assume that $m_{ij}<M_{ij}$ holds for any pair $i<j$, and allow $m_{ij}=-\infty$ and $M_{ij}=\infty$ respectively. Hetyei referred to  the solution set of a system of linear inequalities of the form \eqref{Def-Weighted-Digraph} as a {\em weighted digraphical polytope} (see also \cite[Definition 2.1]{Hetyei2024}). The following definition characterizes how to construct a weighted digraph from the  system of linear inequalities as given in \eqref{Def-Weighted-Digraph}.
	\begin{definition}[\cite{Hetyei2024}, Definition 2.3]\label{WD}
		{\rm For every system of linear inequalities \eqref{Def-Weighted-Digraph}, the associated {\em weighted digraph} is constructed as follows: For each $i<j$, if $m_{ij}>-\infty$, we create a directed edge $i\rightarrow j$ with weight $m_{ij}$; if $M_{ij}<\infty$, we create a directed edge $i\leftarrow j$ with weight $-M_{ij}$. An {\em $m$-ascending cycle} in the associated weighted digraph is a directed cycle, along which the sum of the weights is non-negative. The associated weighted digraph is called {\em $m$-acyclic} if it contains no $m$-ascending cycle.
		}
	\end{definition}
	Conversely, the associated weighted digraph uniquely encodes the system \eqref{Def-Weighted-Digraph} as well. Without loss of generality, we may assume $A=\{a_1<a_2<\cdots<a_m\}$ in $\mathbb{R}$. Then every region of the hyperplane arrangement $\mathcal{B}_n^A$ is a weighted digraphical polytope and can be described by a system of linear inequalities \eqref{Def-Weighted-Digraph}, where for each pair $1\le i<j\le n$, $m_{ij}$ is either $-\infty$ or some member of the set $A$ and $M_{ij}$ is given by the following formula:
	\[
	M_{ij}=\begin{cases}
		a_1, & \mbox{if }m_{ij}=-\infty; \\
		a_{k+1}, & \mbox{if  } m_{ij}=a_k \mbox{ for some }  k<m; \\
		\infty, & \mbox{if } m_{ij}=a_m.
	\end{cases}
	\]
	We continue introducing the {\em valid weighted digraph} (see also \cite[Definition 2.14]{Hetyei2024}) .
	\begin{definition}\label{Valid-WD}
		{\rm Let $A=\{a_1<a_2<\cdots<a_m\}$ be a set in $\mathbb{R}$. A weighted digraph on the vertex set $[n]$ is {\em valid} if for each pair $1\le i<j\le n$ satisfying exactly one of the following holds:
			\begin{itemize}
				\item [{\rm (1)}] there is no directed edge $i\rightarrow j$, and there is exactly one directed edge $i\leftarrow j$ of weight $-a_1$;
				\item [{\rm (2)}] there is a directed edge $i\rightarrow j$ of weight $a_k$, and there is a directed edge $i\leftarrow j$ of weight $-a_{k+1}$ for some $1\le k<m$;
				\item [{\rm (3)}] there is exactly one directed edge $i\rightarrow j$ of weight $a_m$, and there is no directed edge $i\leftarrow j$.
			\end{itemize}
		}
	\end{definition}
	
	Notice that the regions of level $1$ or the relatively bounded regions of  any deformation $\mathcal{A}$ of the braid arrangement $\mathcal{B}_n$ correspond precisely to the bounded regions of the corresponding essential arrangement ${\rm ess}(\mathcal{A}):=\{H\cap V_{n-1}:H\in\mathcal{A}\}$, where $V_{n-1}$ is the subspace in $\mathbb{R}^n$ of all vectors $(x_1,x_2,\ldots,x_n)$ such that $x_1+x_2+\cdots+x_n=0$. Therefore, \cite[Corollary 2.15]{Hetyei2024} actually establishes a bijection between the (relatively bounded) regions of any deformation of the braid arrangement and the (strongly connected) valid $m$-acyclic weighted digraphs. The next result is as a special case of this work.
	\begin{proposition}[\cite{Hetyei2024}, Corollary 2.15] \label{Bij-1}
		Let $A=\{a_1<a_2<\cdots<a_m\}$ be a set in $\mathbb{R}$. Then the regions of $\mathcal{B}_n^A$ are in bijection with the valid $m$-acyclic weighted digraphs on the vertex set $[n]$ in such a way that the relatively bounded regions correspond to the strongly connected valid $m$-acyclic weighted digraphs.
	\end{proposition}
	
	Importantly, the $m$-acyclic property can be verified independently for each strong component of the weighted digraph. Building on this property, Hetyei further provided a structural theorem in \cite[Theorem 2.16]{Hetyei2024}, which illustrates that the associated weighted digraph of every region in any deformation of the braid arrangement $\mathcal{B}_n$ can be uniquely determined and constructed from strongly connected valid $m$-acyclic weighted digraphs. The restriction of  \cite[Theorem 2.16]{Hetyei2024} to $\mathcal{B}_n^A$ is the following statement.
	\begin{proposition}[\cite{Hetyei2024}, Theorem 2.16]\label{Bij-2}
		Let $A=\{a_1<a_2<\cdots<a_m\}$ be a set in $\mathbb{R}$. Then the associated weighted digraph of any region of $\mathcal{B}_n^A$ may be uniquely constructed as follows:
		\begin{itemize}
			\item [\rm{(1)}] For a fixed ordered set partition $(B_1,B_2,\ldots,B_l)$ of the set $[n]$, the parts of the ordered set partition correspond to the vertex sets of the strong components. For any pair $(i, j)$ having the property that the part containing $i$ precedes the part containing $j$, there exists a directed edge $i\rightarrow j$ with a corresponding weight, but no directed edge $i\leftarrow j$ is present. Specifically, the weight on $i\rightarrow j$ is $a_m$ if $i<j$ and $-a_1$ if $i>j$.
			\item [\rm{(2)}] For each strong component on the vertex set $B_i$, we independently select a strongly connected valid $m$-acyclic weighted digraph.
		\end{itemize}
	\end{proposition}
	
	\subsection{Proof of \autoref{Main1}}
	Before proving \autoref{Main1}, the following lemma is also required.
	\begin{lemma}\label{Lem-1}
		Let $X$ and $Y$ be two sets in $\mathbb{R}^n$. If there exists $r>0$ such that
		\[
		X\subseteq\big\{\bm x \in \mathbb{R}^n : \min_{\bm y \in Y} \|\bm x - \bm y\| \le r \big\},
		\]
		then $l(X)\le l(Y)$.
	\end{lemma}
	\begin{proof}
		Suppose $W$ is the minimal dimensional subspace of $\mathbb{R}^n$ such that there exists $r' > 0$ satisfying
		\begin{equation}\label{EQ2-1}
			Y\subseteq \big\{\bm x \in \mathbb{R}^n : \min_{\bm w \in W} \|\bm x - \bm w\| \le r' \big\}.
		\end{equation}
		The condition $X\subseteq\big\{\bm x \in \mathbb{R}^n : \min_{\bm y \in Y} \|\bm x - \bm y\| \le r \big\}$ implies that for any $\bm x \in X$, there exists $\bm y \in Y$ such that $\| \bm x - \bm y \|\le r$.
		It follows from \eqref{EQ2-1} that for such $\bm y \in Y$, there is $\bm w\in W$ for which $\|\bm y-\bm w\|\le r'$. Then we have
		\[
		\|\bm x - \bm w \|\le \|\bm x-\bm y\| +\|\bm y-\bm w \|\le r+r'.
		\]
		This means
		\[
		X\subseteq\big\{\bm x \in \mathbb{R}^n : \min_{\bm w \in W} \|\bm x - \bm w\| \le r+r' \big\}.
		\]
		Therefore, $l(X)\le \dim W =l(Y)$, which finishes the proof.
	\end{proof}
	
	For convenience, we denote by $D(R)$ the associated $m$-acyclic weighted digraph of any region $R$ of $\mathcal{B}_n^A$.
	\begin{lemma}\label{Lem-2}
		Let $A=\{a_1<\cdots<a_m\}$ be a set in $\mathbb{R}$. Given a region $R$ of $\mathcal{B}_n^A$ and a point $\bm x=(x_1,\ldots,x_n)\in R$. If $i$ and $j$ belong to some strong component $N$ of $D(R)$, then we have
		\[
		|x_i-x_j|\le|V(N)|^2a,
		\]
		where $a=\max\{|a_1|,|a_m|\}$.
	\end{lemma}
	\begin{proof}
		Since $D(R)$ is a valid $m$-acyclic weighted digraph, according to the definition of the valid $m$-acyclic weighted digraph in \autoref{Valid-WD}, for any directed edge $h\rightarrow g$ in $D(R)$, we have the following properties:
		\begin{equation}\label{Basic-Fact}
			x_h-x_g>a_q \mbox{ if } h<g \quad\mbox{ and }\quad x_h-x_g>-a_q \mbox{ if } h>g \mbox{ for some } q\in[m].
		\end{equation}
		As $i$ and $j$ belong to the same strong component $N$, there is a directed cycle $C=i_0\rightarrow i_1\rightarrow\cdots\rightarrow i_{s-1}\rightarrow i_s\rightarrow i_{s+1}\rightarrow\cdots\rightarrow i_t\rightarrow i_0$ in $N$ containing the vertices $i$ and $j$, where $i_0=i$ and $i_s=j$. Now, we consider the directed edge $i_p\rightarrow i_{p+1}$ in the directed cycle $C$.  When $i_p<i_{p+1}$, we have $a_q<x_{i_p}-x_{i_{p+1}}$ for some $q\in[m]$ via \eqref{Basic-Fact}. If $q<m$, then $a_q<x_{i_p}-x_{i_{p+1}}<a_{q+1}$. If $q=m$, using \eqref{Basic-Fact} again, the directed path $i_{p+1}\rightarrow i_{p+2}\rightarrow\cdots\rightarrow i_p$ in $C$ implies that
		\begin{align*}
			&x_{i_{p+1}}-x_{i_p}=(x_{i_{p+1}}-x_{i_{p+2}})+(x_{i_{p+2}}-x_{i_{p+3}})+\cdots+(x_{i_{p-1}}-x_{i_p})\\
			&\ge b_{p+1}+b_{p+2}+\cdots+b_{p-1},
		\end{align*}
		where each $b_k$ represents $a_w$ or $-a_w$ for some $w\in[m]$. Hence, we obtain
		\[
		-a\le a_m<x_{i_p}-x_{i_{p+1}}<-(b_{p+1}+b_{p+2}+\cdots+b_{p-1})\le(t+1)a
		\]
		in this case. We conclude $-a<x_{i_p}-x_{i_{p+1}}<(t+1)a$ in the case that $i_p<i_{p+1}$. Likewise, we can prove $-a<x_{i_p}-x_{i_{p+1}}<(t+1)a$ when $i_p>i_{p+1}$ as well. Therefore, we deduce
		\[
		-sa<x_i-x_j=x_{i_0}-x_{i_s}=(x_{i_0}-x_{i_1})+(x_{i_1}-x_{i_2})+\cdots+(x_{i_{s-1}}-x_{i_s})<s(t+1)a.
		\]
		So $|x_i-x_j|<s(t+1)a\le |V(N)|^2a$, which finishes the proof.
	\end{proof}
	The following result states that the level of any region of $\mathcal{B}_n^A$ is precisely the number of strong components of the associated $m$-acyclic weighted digraph.
	\begin{theorem}\label{Main-Lel}
		Let $A=\{a_1,\ldots,a_m\}$ consist of distinct real numbers. For any region $R$ of $\mathcal{B}_n^A$, the level of $R$ equals the number of strong components of $D(R)$.
	\end{theorem}
	\begin{proof}
		Without loss of generality, suppose $A=\{a_1<a_2<\cdots<a_m\}$ in $\mathbb{R}$. From \autoref{Bij-2}, we may assume that the vertex sets of strong components of $D(R)$ induce an ordered set partition $(B_1,B_2,\ldots,B_l)$ of the set $[n]$ satisfying the properties outlined in \autoref{Bij-2}. To prove $l(R)\le l$, we construct a subspace $W(R)$ in $\mathbb{R}^n$ as follows:
		\[
		W(R):=\big\{\bm x=(x_1,\ldots,x_n)\in\mathbb{R}^n:x_i=x_j \mbox{ if } i,j\in B_k \mbox{ for some }1\le k\le l\big\}.
		\]
		Clearly the subspace $W(R)$ has dimension $l$ and the level $l\big(W(R)\big)=l$. Therefore, from \autoref{Lem-1}, we only need to show the following relation
		\[
		R\subseteq \big\{\bm x \in \mathbb{R}^n : \min_{\bm w \in W(R)} \|\bm x - \bm w\| \le r \big\}
		\]
		for some positive real number $r$. Without loss of generality, for each $k=1,2,\ldots,l$, suppose $B_k=\{i_{k-1}+1,\ldots,i_k\}$ with $i_0=0$, $i_{k-1}+1\le i_k\le n$ and $i_l=n$. For any point $\bm x=(x_1,\ldots,x_n)\in R$, we take $\bm w=(w_1,\ldots,w_n)\in W(R)$ satisfying
		\[
		w_{i_{k-1}+1}=w_{i_{k-1}+2}=\cdots=w_{i_k}=x_{i_k}, \quad k=1,2,\ldots,l.
		\]
		By \autoref{Lem-2}, we have
		\[
		|x_i-x_{i_k}|\le(i_k-i_{k-1})^2a\le n^2a
		\]
		for any $i\in B_k$, where $a=\max\{|a_1|,|a_m|\}$. Then we deduce
		\begin{align*}
			\|\bm x-\bm w\|&=\sqrt{\sum_{k=1}^l\sum_{i=i_{k-1}+1}^{i_k}(x_i-x_{i_k})^2}\le n^2\sqrt na.
		\end{align*}
		So, we obtain $l(R)\le l$.
		
		On the other hand, to verify $l(R)\ge l$, we construct a cone $C(R)$ in $\mathbb{R}^n$ as follows:
		\[
		C(R):=\big\{\bm x=(x_1,\ldots,x_n)\in W(R): x_i\ge x_j \mbox{ for } i\in B_k, j\in B_s \mbox{ with }1\le k<s\le n\big\}.
		\]
		Obviously, the cone $C(R)$ has dimension $l$ and the level $l\big(C(R)\big)=l$. Given a point $\bm z=(z_1,\ldots, z_n)\in R$, we assert
		\[
		\bm z+C(R) :=\big\{\bm z+\bm x =(z_1+x_1,\ldots, z_n+x_n) : \bm x=(x_1,\ldots, x_n)\in C(R) \big\}\subseteq R.
		\]
		Notably, for any point $\bm x\in C(R)$, we have that
		\begin{itemize}
			\item [{\rm (a)}] if $i$ and $j$ belong to the same part, then $\left( z_i+x_i\right)-\left(z_j+x_j \right)=z_i-z_j$;
			\item [{\rm (b)}] if the part containing $i$ precedes the part containing $j$, then
			\[
			(x_i+z_i)-(x_j+z_j)\ge z_i-z_j>a_m\quad\mbox{if}\quad i<j
			\]
			and
			\[
			(x_i+z_i)-(x_j+z_j)\ge z_i-z_j>-a_1\mbox{ if }i>j\mbox{ via the property (1) in \autoref{Bij-2}}.
			\]
		\end{itemize}
		Hence, $\bm z+\bm x\in R$ for all $\bm x\in C(R)$, that is, $\bm z+C(R)\subseteq R$. We further derive
		\[
		l=l\big(C(R)\big)=l\big(\bm z+C(R)\big)\le l(R).
		\]
		So we obtain $l(R)=l$, which completes the proof.
	\end{proof}
	As a direct consequence of \autoref{Bij-1} and \autoref{Main-Lel} , we have the next result. In fact, \autoref{Bij-3} can be generalized to any deformation of the braid arrangement as the form \eqref{Deformation}. However, a similar but more complicated discussion will be omitted here and left for interested readers to explore.
	\begin{corollary}\label{Bij-3}
		Let $A=\{a_1<a_2<\cdots<a_m\}$ be a set in $\mathbb{R}$. For a fixed positive integer $l$, the regions of level $l$ of $\mathcal{B}_n^A$ are in bijection with the valid $m$-acyclic weighted digraphs on the vertex set $[n]$ that have exactly $l$ strong components.
	\end{corollary}
	Below we provide a small example to explain \autoref{Bij-3}.
	\begin{example}\label{Exa1}
		{\rm Let $\mathcal{B}_3^{[1,2]}=\big\{H_{ij}^k:x_i-x_j=k\mid 1\le i<j\le3,k=1,2\big\}$ be a hyperplane arrangement in $\mathbb{R}^3$, as illustrated in \autoref{Fig1}.  In \autoref{Fig1}, we describe the one-to-one correspondence between the regions of $\mathcal{B}_3^{[1,2]}$ and the associated valid $m$-acyclic weighted digraphs on the vertex set $[3]$. Specifically, every valid $m$-acyclic weighted digraph is precisely placed within the corresponding region of $\mathcal{B}_n^{[1,2]}$ in \autoref{Fig1}, where the blue, red and green valid $m$-acyclic weighted digraphs having $1$, $2$ and $3$ strong components in turn,  correspond to the regions of level $1$, $2$ and $3$ respectively.
			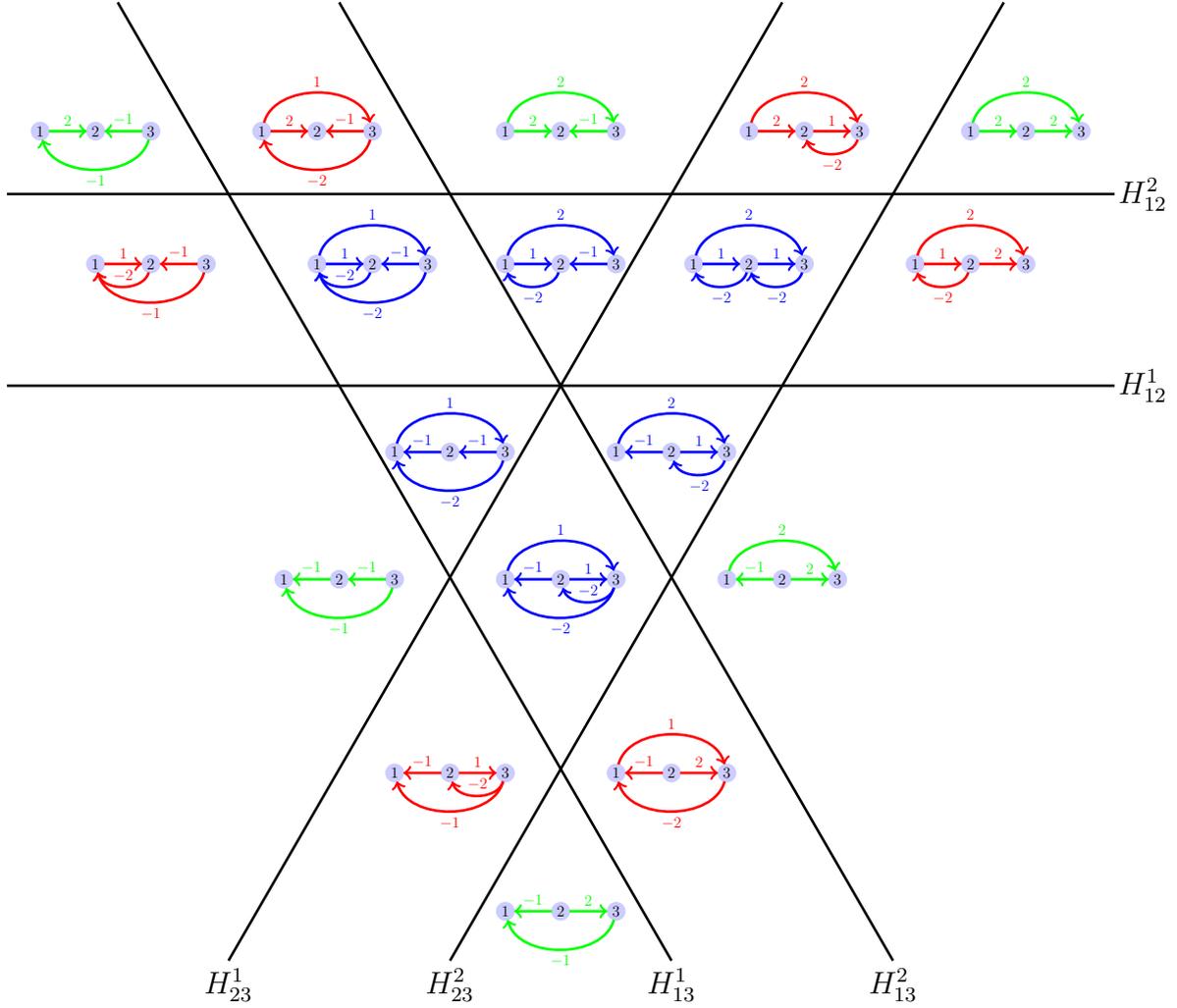
\begin{figure}[H]
				\centering
				\begin{tikzpicture}[scale=1.5,line width=1pt]
					\pgfmathsetmacro{\sqrtthree}{sqrt(3)}
					\pgfmathsetmacro{\twosqrtthree}{2*sqrt(3)}
					\pgfmathsetmacro{\threesqrtthree}{3*sqrt(3)}
					\draw (-5, 0) -- (5, 0) node[anchor=west,shift={(-0.1,0)}] {$H_{12}^1$};
					\draw(-5, \sqrtthree) -- (5,\sqrtthree) node[anchor= west,shift={(-0.1,0)}] {$H_{12}^2$};
					\draw (-4,\twosqrtthree) -- (1,-\threesqrtthree) node[anchor=south,shift={(0,-0.7)}] {$H_{13}^1$};
					\draw (-2, \twosqrtthree) -- (3,-\threesqrtthree) node[anchor=south,shift={(0,-0.7)}] {$H_{13}^2$};
					\draw (2, \twosqrtthree) -- (-3, -\threesqrtthree) node[anchor=south,shift={(0, -0.7)}] {$H_{23}^1$};
					\draw (4,\twosqrtthree) -- (-1, -\threesqrtthree) node[anchor=south,shift={(0,-0.7)}] {$H_{23}^2$};
					
					\node[circle, blue!20, draw, fill=blue!20, scale=0.5, inner sep=1pt] (v1) at (-4.7,2.3) {$\black 1$}; \node[circle, blue!20,draw, fill=blue!20, scale=0.5, inner sep=1pt]  (u1) at   (-4.2,2.3) {$\black 2$}; \node[circle,blue!20, draw, fill=blue!20,scale=0.5, inner sep=1pt]  (w1) at  (-3.7,2.3) {$\black 3$};
					\draw[green,-to] (v1) --node[scale=0.5, pos=0.5, above] {$2$} (u1) ;
					\draw[green,-to] (w1) to [bend left=75] node[scale=0.5, pos=0.5, below] {$-1$} (v1);
					\draw[green,-to] (w1)--node[scale=0.5, pos=0.5, above] {$-1$}(u1);
					
					\node[circle, blue!20, draw, fill=blue!20, scale=0.5, inner sep=1pt] (v2) at (-2.7,2.3) {$\black 1$}; \node[circle, blue!20,draw, fill=blue!20, scale=0.5, inner sep=1pt]  (u2) at   (-2.2,2.3) {$\black 2$}; \node[circle,blue!20, draw, fill=blue!20,scale=0.5, inner sep=1pt]  (w2) at  (-1.7,2.3) {$\black 3$};
					\draw[red,-to] (v2) --node[scale=0.5, pos=0.5, above] {$2$} (u2) ;
					\draw[red,-to] (v2) to [bend left=75] node[scale=0.5, pos=0.5, above] {$1$} (w2);\draw[red,-to] (w2) to [bend left=75] node[scale=0.5, pos=0.5, below] {$-2$} (v2);
					\draw[red,-to] (w2)--node[scale=0.5, pos=0.5, above] {$-1$}(u2);
					
					\node[circle, blue!20, draw, fill=blue!20, scale=0.5, inner sep=1pt] (v3) at (-0.5,2.3) {$\black 1$}; \node[circle, blue!20,draw, fill=blue!20, scale=0.5, inner sep=1pt]  (u3) at   (0,2.3) {$\black 2$}; \node[circle,blue!20, draw, fill=blue!20,scale=0.5, inner sep=1pt]  (w3) at  (0.5,2.3) {$\black 3$};
					\draw[green,-to] (v3) --node[scale=0.5, pos=0.5, above] {$2$} (u3) ;
					\draw[green,-to] (v3) to [bend left=75] node[scale=0.5, pos=0.5, above] {$2$} (w3);
					\draw[green,-to] (w3)--node[scale=0.5, pos=0.5, above] {$-1$}(u3);
					
					\node[circle, blue!20, draw, fill=blue!20, scale=0.5, inner sep=1pt] (v4) at (1.7,2.3) {$\black 1$}; \node[circle, blue!20,draw, fill=blue!20, scale=0.5, inner sep=1pt]  (u4) at   (2.2,2.3) {$\black 2$}; \node[circle,blue!20, draw, fill=blue!20,scale=0.5, inner sep=1pt]  (w4) at  (2.7,2.3) {$\black 3$};
					\draw[red,-to] (v4) --node[scale=0.5, pos=0.5, above] {$2$} (u4) ;
					\draw[red,-to] (v4) to [bend left=75] node[scale=0.5, pos=0.5, above] {$2$} (w4);
					\draw[red,-to] (u4)--node[scale=0.5, pos=0.5, above] {$1$}(w4); \draw[red,-to] (w4) to [bend left=75] node[scale=0.5, pos=0.5, below] {$-2$} (u4);
					
					\node[circle, blue!20, draw, fill=blue!20, scale=0.5, inner sep=1pt] (v5) at (3.7,2.3) {$\black 1$}; \node[circle, blue!20,draw, fill=blue!20, scale=0.5, inner sep=1pt]  (u5) at   (4.2,2.3) {$\black 2$}; \node[circle,blue!20, draw, fill=blue!20,scale=0.5, inner sep=1pt]  (w5) at  (4.7,2.3) {$\black 3$};
					\draw[green,-to] (v5) --node[scale=0.5, pos=0.5, above] {$2$} (u5) ;
					\draw[green,-to] (v5) to [bend left=75] node[scale=0.5, pos=0.5, above] {$2$} (w5);
					\draw[green,-to] (u5)--node[scale=0.5, pos=0.5, above] {$2$}(w5);
					
					\node[circle, blue!20, draw, fill=blue!20, scale=0.5, inner sep=1pt] (v6) at (-4.2,1.1) {$\black 1$}; \node[circle, blue!20,draw, fill=blue!20, scale=0.5, inner sep=1pt]  (u6) at   (-3.7,1.1) {$\black 2$}; \node[circle,blue!20, draw, fill=blue!20,scale=0.5, inner sep=1pt]  (w6) at  (-3.2,1.1) {$\black 3$};
					\draw[red,-to] (v6) --node[scale=0.5, pos=0.5, above] {$1$} (u6) ; \draw[red,-to] (u6) to [bend left=75] node[scale=0.5, pos=0.5, above] {$-2$} (v6); \draw[red,-to] (w6) to [bend left=75] node[scale=0.5, pos=0.5, below] {$-1$} (v6);
					\draw[red,-to] (w6)--node[scale=0.5, pos=0.5, above] {$-1$}(u6);
					
					\node[circle, blue!20, draw, fill=blue!20, scale=0.5, inner sep=1pt] (v7) at (-2.2,1.1) {$\black 1$}; \node[circle, blue!20,draw, fill=blue!20, scale=0.5, inner sep=1pt]  (u7) at   (-1.7,1.1) {$\black 2$}; \node[circle,blue!20, draw, fill=blue!20,scale=0.5, inner sep=1pt]  (w7) at  (-1.2,1.1) {$\black 3$};
					\draw[blue,-to] (v7) --node[scale=0.5, pos=0.5, above] {$1$} (u7) ; \draw[blue,-to] (u7) to [bend left=75] node[scale=0.5, pos=0.5, above] {$-2$} (v7); \draw[blue,-to] (v7) to [bend left=75] node[scale=0.5, pos=0.5, above] {$1$} (w7); \draw[blue,-to] (w7) to [bend left=75] node[scale=0.5, pos=0.5, below] {$-2$} (v7);
					\draw[blue,-to] (w7)--node[scale=0.5, pos=0.5, above] {$-1$}(u7);
					
					\node[circle, blue!20, draw, fill=blue!20, scale=0.5, inner sep=1pt] (v8) at (-0.5,1.1) {$\black 1$}; \node[circle, blue!20,draw, fill=blue!20, scale=0.5, inner sep=1pt]  (u8) at   (-0,1.1) {$\black 2$}; \node[circle,blue!20, draw, fill=blue!20,scale=0.5, inner sep=1pt]  (w8) at  (0.5,1.1) {$\black 3$}; \draw[blue,-to] (v8) --node[scale=0.5, pos=0.5, above] {$1$} (u8) ; \draw[blue,-to] (u8) to [bend left=75] node[scale=0.5, pos=0.5, below] {$-2$} (v8); \draw[blue,-to] (v8) to [bend left=75] node[scale=0.5, pos=0.5, above] {$2$} (w8);\draw[blue,-to] (w8)--node[scale=0.5, pos=0.5, above] {$-1$}(u8);
					
					\node[circle, blue!20, draw, fill=blue!20, scale=0.5, inner sep=1pt] (v9) at (1.2,1.1) {$\black 1$}; \node[circle, blue!20,draw, fill=blue!20, scale=0.5, inner sep=1pt]  (u9) at   (1.7,1.1) {$\black 2$}; \node[circle,blue!20, draw, fill=blue!20,scale=0.5, inner sep=1pt]  (w9) at  (2.2,1.1) {$\black 3$}; \draw[blue,-to] (v9) --node[scale=0.5, pos=0.5, above] {$1$} (u9) ; \draw[blue,-to] (u9) to [bend left=75] node[scale=0.5, pos=0.5, below] {$-2$} (v9); \draw[blue,-to] (v9) to [bend left=75] node[scale=0.5, pos=0.5, above] {$2$} (w9);\draw[blue,-to] (u9)--node[scale=0.5, pos=0.5, above] {$1$}(w9);
					\draw[blue,-to] (w9) to [bend left=75] node[scale=0.5, pos=0.5, below] {$-2$} (u9);
					
					\node[circle, blue!20, draw, fill=blue!20, scale=0.5, inner sep=1pt] (v10) at (3.2,1.1) {$\black 1$}; \node[circle, blue!20,draw, fill=blue!20, scale=0.5, inner sep=1pt]  (u10) at   (3.7,1.1) {$\black 2$}; \node[circle,blue!20, draw, fill=blue!20,scale=0.5, inner sep=1pt]  (w10) at  (4.2,1.1) {$\black 3$};
					\draw[red,-to] (v10) --node[scale=0.5, pos=0.5, above] {$1$} (u10) ; \draw[red,-to] (u10) to [bend left=75] node[scale=0.5, pos=0.5, below] {$-2$} (v10); \draw[red,-to] (v10) to [bend left=75] node[scale=0.5, pos=0.5, above] {$2$} (w10);
					\draw[red,-to] (u10)--node[scale=0.5, pos=0.5, above] {$2$}(w10);
					
					\node[circle, blue!20, draw, fill=blue!20, scale=0.5, inner sep=1pt] (v11) at (-2.5,-1.75) {$\black 1$}; \node[circle, blue!20,draw, fill=blue!20, scale=0.5, inner sep=1pt]  (u11) at   (-2,-1.75) {$\black 2$}; \node[circle,blue!20, draw, fill=blue!20,scale=0.5, inner sep=1pt]  (w11) at  (-1.5,-1.75) {$\black 3$}; \draw[green,-to] (u11) --node[scale=0.5, pos=0.5, above] {$-1$} (v11) ;
					\draw[green,-to] (w11) to [bend left=75] node[scale=0.5, pos=0.5, below] {$-1$} (v11);
					\draw[green,-to] (w11)--node[scale=0.5, pos=0.5, above] {$-1$}(u11);
					
					\node[circle, blue!20, draw, fill=blue!20, scale=0.5, inner sep=1pt] (v12) at (-1.5,-0.6) {$\black 1$}; \node[circle, blue!20,draw, fill=blue!20, scale=0.5, inner sep=1pt]  (u12) at   (-1,-0.6) {$\black 2$}; \node[circle,blue!20, draw, fill=blue!20,scale=0.5, inner sep=1pt]  (w12) at  (-0.5,-0.6) {$\black 3$}; \draw[blue,-to] (u12) --node[scale=0.5, pos=0.5, above] {$-1$} (v12) ;
					\draw[blue,-to] (v12) to [bend left=75] node[scale=0.5, pos=0.5, above] {$1$} (w12);\draw[blue,-to] (w12) to [bend left=75] node[scale=0.5, pos=0.5, below] {$-2$} (v12);
					\draw[blue,-to] (w12)--node[scale=0.5, pos=0.5, above] {$-1$}(u12);

					\node[circle, blue!20, draw, fill=blue!20, scale=0.5, inner sep=1pt] (v13) at (-0.5,-1.75) {$\black 1$}; \node[circle, blue!20,draw, fill=blue!20, scale=0.5, inner sep=1pt]  (u13) at   (0,-1.75) {$\black 2$}; \node[circle,blue!20, draw, fill=blue!20,scale=0.5, inner sep=1pt]  (w13) at  (0.5,-1.75) {$\black 3$}; \draw[blue,-to] (u13) --node[scale=0.5, pos=0.5, above] {$-1$} (v13) ;
					\draw[blue,-to] (v13) to [bend left=75] node[scale=0.5, pos=0.5, above] {$1$} (w13);\draw[blue,-to] (w13) to [bend left=75] node[scale=0.5, pos=0.5, below] {$-2$} (v13);
					\draw[blue,-to] (u13)--node[scale=0.5, pos=0.5, above] {$1$}(w13); \draw[blue,-to] (w13) to [bend left=75] node[scale=0.5, pos=0.5, above] {$-2$} (u13);
					
					\node[circle, blue!20, draw, fill=blue!20, scale=0.5, inner sep=1pt] (v14) at (0.5,-0.6) {$\black 1$}; \node[circle, blue!20,draw, fill=blue!20, scale=0.5, inner sep=1pt]  (u14) at   (1,-0.6) {$\black 2$}; \node[circle,blue!20, draw, fill=blue!20,scale=0.5, inner sep=1pt]  (w14) at  (1.5,-0.6) {$\black 3$}; \draw[blue,-to] (u14) --node[scale=0.5, pos=0.5, above] {$-1$} (v14) ;
					\draw[blue,-to] (v14) to [bend left=75] node[scale=0.5, pos=0.5, above] {$2$} (w14);
					\draw[blue,-to] (u14)--node[scale=0.5, pos=0.5, above] {$1$}(w14); \draw[blue,-to] (w14) to [bend left=75] node[scale=0.5, pos=0.5, below] {$-2$} (u14);
					
					\node[circle, blue!20, draw, fill=blue!20, scale=0.5, inner sep=1pt] (v15) at (1.5,-1.75) {$\black 1$}; \node[circle, blue!20,draw, fill=blue!20, scale=0.5, inner sep=1pt]  (u15) at   (2,-1.75) {$\black 2$}; \node[circle,blue!20, draw, fill=blue!20,scale=0.5, inner sep=1pt]  (w15) at  (2.5,-1.75) {$\black 3$}; \draw[green,-to] (u15) --node[scale=0.5, pos=0.5, above] {$-1$} (v15) ;
					\draw[green,-to] (v15) to [bend left=75] node[scale=0.5, pos=0.5, above] {$2$} (w15);
					\draw[green,-to] (u15)--node[scale=0.5, pos=0.5, above] {$2$}(w15);
					
					\node[circle, blue!20, draw, fill=blue!20, scale=0.5, inner sep=1pt] (v16) at (-1.5,-3.5) {$\black 1$}; \node[circle, blue!20,draw, fill=blue!20, scale=0.5, inner sep=1pt]  (u16) at   (-1,-3.5) {$\black 2$}; \node[circle,blue!20, draw, fill=blue!20,scale=0.5, inner sep=1pt]  (w16) at  (-0.5,-3.5) {$\black 3$}; \draw[red,-to] (u16) --node[scale=0.5, pos=0.5, above] {$-1$} (v16) ;
					\draw[red,-to] (w16) to [bend left=75] node[scale=0.5, pos=0.5, below] {$-1$} (v16);
					\draw[red,-to] (u16)--node[scale=0.5, pos=0.5, above] {$1$}(w16); \draw[red,-to] (w16) to [bend left=75] node[scale=0.5, pos=0.5, above] {$-2$} (u16);
					
					\node[circle, blue!20, draw, fill=blue!20, scale=0.5, inner sep=1pt] (v17) at (-0.5,-4.75) {$\black 1$}; \node[circle, blue!20,draw, fill=blue!20, scale=0.5, inner sep=1pt]  (u17) at   (0,-4.75) {$\black 2$}; \node[circle,blue!20, draw, fill=blue!20,scale=0.5, inner sep=1pt]  (w17) at  (0.5,-4.75) {$\black 3$}; \draw[green,-to] (u17) --node[scale=0.5, pos=0.5, above] {$-1$} (v17) ;
					\draw[green,-to] (w17) to [bend left=75] node[scale=0.5, pos=0.5, below] {$-1$} (v17);
					\draw[green,-to] (u17)--node[scale=0.5, pos=0.5, above] {$2$}(w17);

					\node[circle, blue!20, draw, fill=blue!20, scale=0.5, inner sep=1pt] (v18) at (0.5,-3.5) {$\black 1$}; \node[circle, blue!20,draw, fill=blue!20, scale=0.5, inner sep=1pt]  (u18) at   (1,-3.5) {$\black 2$}; \node[circle,blue!20, draw, fill=blue!20,scale=0.5, inner sep=1pt]  (w18) at  (1.5,-3.5) {$\black 3$}; \draw[red,-to] (u18) --node[scale=0.5, pos=0.5, above] {$-1$} (v18) ;
					\draw[red,-to] (v18) to [bend left=75] node[scale=0.5, pos=0.5, above] {$1$} (w18);\draw[red,-to] (w18) to [bend left=75] node[scale=0.5, pos=0.5, below] {$-2$} (v18);
					\draw[red,-to] (u18)--node[scale=0.5, pos=0.5, above] {$2$}(w18);
					
				\end{tikzpicture}
				\caption{Regions of $\mathcal{B}_3^{[1,2]}$ labeled by their associated $m$-acyclic weighted digraphs}\label{Fig1}
			\end{figure}
		}
	\end{example}
	
	We are now ready to prove \autoref{Main1}.
	\begin{proof}[Proof of \autoref{Main1}]
		Without loss of generality, let $A=\{a_1<a_2<\cdots<a_m\}$ in $\mathbb{R}$. When $l=0$, the case is trivial. For $l\ge 1$, applying \autoref{Bij-3}, there is a one-to-one correspondence between the regions of level $l$ of $\mathcal{B}_n^A$ and the valid $m$-acyclic weighted digraphs on the vertex set $[n]$ that have exactly $l$ strong components. From \autoref{Bij-2}, any such valid $m$-acyclic weighted digraph can be constructed in the following way: Let $\{B_1,B_2,\ldots,B_l\}$ be a set partition of the set $[n]$. Given the size $n_i$ of  every $B_i$, the number of ways to select an ordered set partition is equal to $\binom{n}{n_1,n_2,\ldots,n_l}$. Next, we need to choose a strongly connected valid $m$-acyclic weighted digraph for each part $B_i$. By \autoref{Bij-1}, there are $r_1(\mathcal{B}_{n_i}^A)$ ways to perform this step on the part $B_i$. Therefore, we have
		\begin{equation*}
			r_l\big(\mathcal{B}_n^A\big)=\sum_{\substack{n_1+\cdots+n_l=n,\\n_1,\ldots,n_l\ge 1}}\binom{n}{n_1,\ldots ,n_l}\prod_{i=1}^lr_1\big(\mathcal{B}_{n_i}^A\big).
		\end{equation*}
		Together with \eqref{Def1}, we further deduce
		\begin{align*}
			R_l(A;x)&=\sum_{n\ge l}r_l(\mathcal{B}_n^A)\frac{x^n}{n!}\\
			&=\sum_{n\ge l}\sum_{\substack{n_1+\cdots+n_l=n,\\n_1,\ldots,n_l\ge 1}}\binom{n}{n_1,\ldots ,n_l}\prod_{i=1}^lr_1\big(\mathcal{B}_{n_i}^A\big)\frac{x^n}{n!}\\
			&=\big(R_1(A;x)\big)^l.
		\end{align*}
		This is equivalent to the following relation
		\[
		R_l(A;x)=R_k(A;x)R_{l-k}(A;x).
		\]
		Comparing the coefficient in $\frac{x^n}{n!}$ of $R_l(A;x)$ and $R_k(A;x)R_{l-k}(A;x)$, we arrive at
		\[
		r_l(\mathcal{B}_n^A)=\sum_{i=0}^n\binom{n}{i}r_k\big(\mathcal{B}_i^A\big)r_{l-k}\big(\mathcal{B}_{n-i}^A\big).
		\]
		We complete the proof.
	\end{proof}
	It is important to note that the total number of regions in $\mathcal{B}_n^A$ is the sum of the numbers of regions of level $l$ for all $l=1,\ldots,n$, that is,
	\[
	r(\mathcal{B}_n^A)=\sum_{l=1}^nr_l(\mathcal{B}_n^A).
	\]
	Together with the second formula in \autoref{Main1}, we gain
	\[
	r(\mathcal{B}_n^A)=\sum_{l=1}^{n}\sum_{\substack{n_1+\cdots+n_l=n,\\n_1,\ldots,n_l\ge 1}}\binom{n}{n_1,\ldots ,n_l}\prod_{i=1}^lr_1\big(\mathcal{B}_{n_i}^A\big),
	\]
	which can be found in \cite[(2.7)]{Hetyei2024}. This implies
	\[
	R(A;x)=\sum_{l\ge 0}\big(R_1(A;x)\big)^l,
	\]
	where $R(A;x):=\sum_{n\ge 0}r(\mathcal{B}_n^A)\frac{x^n}{n!}$. It follows from the relation $R_l(A;x)=\big(R_1(A;x)\big)^l$ in \autoref{Main1} that we further derive
	\[
	R(A;x)=\sum_{l\ge 0}R_l(A;x).
	\]
	
	At the end of this section, we examine two specific types of arrangements that generalize the classical Catalan arrangement and semiorder arrangement. Specifically, the {\em semiorder-type arrangement} in $\mathbb{R}^n$ is defined as $\mathcal{C}_{n,A}^*:=\mathcal{B}_n^{A\cup -A}$, and the {\em Catalan-type arrangement} in $\mathbb{R}^n$ is given by $\mathcal{C}_{n,A}:=\mathcal{C}_{n,A}^*\cup\mathcal{B}_n$, where $A=\{a_1<a_2<\cdots<a_m\}$ is a positive real number set. As special cases of \autoref{Main1}, we recover the results of \cite[Theorems 1.2, 1.3, and 1.4]{CWYZ2024}.
	\begin{theorem}[\cite{CWYZ2024}, Theorems 1.2, 1.3, and 1.4]
		Let $\mathcal{A}_n$ be either $\mathcal{C}_{n,A}$ or $\mathcal{C}_{n,A}^*$. For any non-negative integers $n$, $l$ and $k$ with $k\le l$, we have
		\[
		R_l(A;x)=\big(R_1(A;x)\big)^l
		\]
		and
		\[
		r_l(\mathcal{A}_n)=\sum_{i=0}^n\binom{n}{i}r_k(\mathcal{A}_i)r_{l-k}(\mathcal{A}_{n-i}).
		\]
	\end{theorem}
	\section{Proof of \autoref{Main2}}\label{Sec-3}
	In this section, we focus on the proof of \autoref{Main2}, which establishes a convolution formula connecting the characteristic polynomial $\chi_{\mathcal{B}_n^A}(t)$ of $\mathcal{B}_n^A$ to binomial coefficients, with the coefficients corresponding to the number of regions of each level $l$. The {\em characteristic polynomial} $\chi_{\mathcal{A}}(t)$ of a hyperplane arrangement $\mathcal{A}$ is defined as
	\[
	\chi_{\mathcal{A}}(t):=\sum_{\mathcal{B}\subseteq\mathcal{A},\,\bigcap_{H\in\mathcal{B}}H\ne\emptyset}(-1)^{|\mathcal{B}|}t^{\dim(\bigcap_{H\in\mathcal{B}}H)},
	\]
	see \cite{Stanley2007}. Stanley \cite{Stanley1996} in 1996  studied the {\em exponential sequence of arrangements}, which is a sequence $\mathfrak{A}=(\mathcal{A}_{1},\mathcal{A}_{2},\ldots)$ of hyperplane arrangements that satisfies the three conditions:
	\begin{enumerate}
		\item[{\rm (a)}] $\mathcal{A}_{n}$ is in $\mathbb{F}^{n}$ for some field $\mathbb{F}$ (independent of $n$).
		\item[{\rm (b)}] Every $H\in \mathcal{A}_{n}$ is parallel to some hyperplane $H'$ in the braid arrangement $\mathcal{B}_{n}$ (over $\mathbb{F}$).
		\item[{\rm (c)}] Let $S$ be a $k$-element subset of $[n]$, and define
		\[\mathcal{A}_{n}^{S}=\{H\in \mathcal{A}_{n}: H \; \text{is parallel to} \; x_{i}-x_{j}=0 \;\text{for some}\; i,j \in S\}.\]
		Then $L(\mathcal{A}_{n}^{S}) \cong L(\mathcal{A}_{k})$.
	\end{enumerate}
	
	Let $\mathcal{A}$ be a hyperplane arrangement in $\mathbb{R}^n$. By abusing the notation, we define
	\[
	r(\mathcal{A}):=(-1)^n\chi_{\mathcal{A}}(-1).
	\]
	Zaslavsky \cite{Zaslavsky1975} demonstrated that this agrees with the definition of $r(\mathcal{A})$ as the number of regions of $\mathcal{A}$. \cite[Theorem 1.2]{Stanley1996} provided a fundamental result related to the exponential sequence of arrangements using the exponential formula \cite{Stanley2024}, Whitney's formula \cite{Whitney1932} and  Zaslavsky's formula \cite{Zaslavsky1975}.
	\begin{theorem}[\cite{Stanley1996}, Theorem 1.2]\label{ESA}
		Let $\mathfrak{A}=(\mathcal{A}_{1},\mathcal{A}_{2},\ldots)$ be an exponential sequence of arrangements. Then
		\[
		\sum_{n\ge 0}\chi_{\mathcal{A}_n}(t)\frac{x^n}{n!}=\Big(\sum_{n\ge 0}(-1)^nr(\mathcal{A}_n)\frac{x^n}{n!}\Big)^{-t}
		\]
		with the convention $\chi_{\mathcal{A}_0}(t)=1$.
	\end{theorem}
	
	Note that the sequence $\mathfrak{B}$ of hyperplane arrangements $\mathcal{B}_1^A,\mathcal{B}_2^A,\ldots$ is an exponential sequence. Below we give a proof of \autoref{Main2} by \autoref{ESA}.
	\begin{proof}[Proof of \autoref{Main2}]
		For the exponential sequence $\mathfrak{B}=(\mathcal{B}_1^A,\mathcal{B}_2^A,\ldots)$, we have
		\begin{equation}\label{EQ1}
			\sum_{n\ge 0}\chi_{\mathcal{B}_n^A}(1)\frac{x^n}{n!}=\Big(\sum_{n\ge 0}(-1)^nr(\mathcal{B}_n^A)\frac{x^n}{n!}\Big)^{-1}
		\end{equation}
		by \autoref{ESA}. The famous Zaslavsky's Formula \cite[Theorem C]{Zaslavsky1975} states that $\chi_{\mathcal{B}_n^A}(1)=(-1)^{n-1}r_1(\mathcal{B}_n^A)$ for any positive integer $n$. Substituting this into \eqref{EQ1} yields
		\[
		1+\sum_{n\ge 1}(-1)^{n-1}r_1(\mathcal{B}_n^A)\frac{x^n}{n!}=\Big(\sum_{n\ge 0}(-1)^nr(\mathcal{B}_n^A)\frac{x^n}{n!}\Big)^{-1}.
		\]
		Combining $R_l(A;x)=\sum_{n\ge 0}r_l(\mathcal{B}_n^A)\frac{x^n}{n!}$ in \eqref{Def1} and \autoref{ESA}, we deduce
		\begin{align*}
			\sum_{n \geq 0}\chi_{\mathcal{B}_n^A}(t)\frac{x^{n}}{n!}
			&=\Big(1+ \sum_{n \geq 1} (-1)^{n-1}r_1(\mathcal{B}_n^A)\frac{x^{n}}{n!}\Big)^t\\
			&=\big(1-R_1(A;-x)\big)^t\\
			&=\sum_{l\geq 0}(-1)^{l}\binom{t}{l}R_{1}(A;-x)^l.
		\end{align*}
		Together with \autoref{Main1}, we obtain
		\[
		\sum_{n \geq 0}\chi_{\mathcal{B}_n^A}(t)\frac{x^{n}}{n!}=\sum_{l \geq 0}(-1)^{l}\binom{t}{l}R_l(A;-x).
		\]
		Since $r_l(\mathcal{B}_n^A)=0$ whenever $l>n$, we further derive from \eqref{Def1} that
		\begin{align*}
			\sum_{n \geq 0}\chi_{\mathcal{B}_n^A}(t)\frac{x^{n}}{n!}&=\sum_{l \geq 0}\sum_{n\ge l}(-1)^{l}r_l(\mathcal{B}_n^A)\binom{t}{l}\frac{(-x)^n}{n!}\\
			&=\sum_{n \geq 0}\left(\sum_{0\le l\le n}(-1)^{n-l}r_l(\mathcal{B}_n^A)\binom{t}{l}\right)\frac{x^n}{n!}.
		\end{align*}
		Comparing the coefficients of $\frac{x^n}{n!}$ on the both sides of the above equation, we obtain
		\[
		\chi_{\mathcal{B}_n^A}(t)=\sum_{l=0}^n(-1)^{n-l}r_l(\mathcal{B}_n^A)\binom{t}{l}.
		\]
		This completes the proof.
	\end{proof}
	
	We revisit the hyperplane arrangement $\mathcal{B}_3^{[1,2]}$ from \autoref{Exa1} to illustrate \autoref{Main2}.
	\begin{example}
		{\rm Let $\mathcal{B}_3^{[1,2]}=\big\{H_{ij}^k:x_i-x_j=k\mid 1\le i<j\le3,k=1,2\big\}$ be a hyperplane arrangement in $\mathbb{R}^3$. The regions of $\mathcal{B}_3^{[1,2]}$ are labeled by their corresponding levels in \autoref{Fig2}. From \autoref{Fig2}, we observe that $r_l(\mathcal{B}_3^{[1,2]})=6$ for $l=1,2,3$ and $r_0(\mathcal{B}_3^{[1,2]})=0$. Therefore, the characteristic polynomial
			\[
			\chi_{\mathcal{B}_3^{[1,2]}}(t)=t^3-6t^2+11t=6\binom{t}{3}-6\binom{t}{2}+6\binom{t}{1}.
			\]
			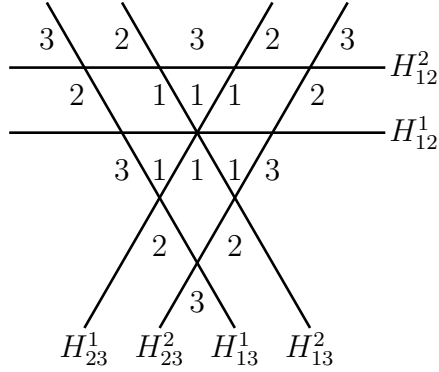
\begin{figure}[H]
				\centering
				\begin{tikzpicture}[scale=0.5,line width=1pt]
					\pgfmathsetmacro{\sqrtthree}{sqrt(3)}
					\pgfmathsetmacro{\twosqrtthree}{2*sqrt(3)}
					\pgfmathsetmacro{\threesqrtthree}{3*sqrt(3)}
					\draw (-5, 0) -- (5, 0) node[anchor=west,shift={(-0.1,0)}] {$H_{12}^1$};
					\draw(-5, \sqrtthree) -- (5,\sqrtthree) node[anchor= west,shift={(-0.1,0)}] {$H_{12}^2$};
					\draw (-4,\twosqrtthree) -- (1,-\threesqrtthree) node[anchor=south,shift={(0,-0.65)}] {$H_{13}^1$};
					\draw (-2, \twosqrtthree) -- (3,-\threesqrtthree) node[anchor=south,shift={(0,-0.65)}] {$H_{13}^2$};
					\draw (2, \twosqrtthree) -- (-3, -\threesqrtthree) node[anchor=south,shift={(0, -0.65)}] {$H_{23}^1$};
					\draw (4,\twosqrtthree) -- (-1, -\threesqrtthree) node[anchor= south,shift={(0,-0.65)}] {$H_{23}^2$};
					\draw (0,1) node{$1$};\draw (1,1) node{$1$};\draw (-1,1) node{$1$};
					\draw (0,-1) node{$1$};\draw (-1,-1) node{$1$};\draw (1,-1) node{$1$};
					\draw (3.2,1) node{$2$};\draw (-3.2,1) node{$2$};\draw (-2,2.5) node{$2$};
					\draw (2,2.5) node{$2$};\draw (1,-3) node{$2$};\draw (-1,-3) node{$2$};
					\draw (4,2.5) node{$3$};\draw (-4,2.5) node{$3$};\draw (2,-1) node{$3$};
					\draw (-2,-1) node{$3$};\draw (0,2.5) node{$3$};\draw (0,-4.5) node{$3$};
				\end{tikzpicture}
				\caption{Regions of $\mathcal{B}_3^{[1,2]}$ labeled by their corresponding levels}\label{Fig2}
			\end{figure}
		}
	\end{example}
	
	We end this section with an application of \autoref{Main2} to the Catalan-type and semiorder-type arrangements, originally presented by Chen et al. \cite{CWYZ2024}.
	\begin{theorem}[\cite{CWYZ2024}, Theorems 1.5]
		Let $\mathcal{A}_n$ be either $\mathcal{C}_{n,A}$ or $\mathcal{C}_{n,A}^*$. Then
		\[\chi_{\mathcal{A}_n}(t)=\sum_{l=0}^n(-1)^{n-l}r_l(\mathcal{A}_n)\binom{t}{l}.\]
	\end{theorem}
	
	\section{Enumeration of $r_l(\mathcal{B}_n^{[-a,b]})$}\label{Sec-4}
	Let $a$ and $b$ be any non-negative integers. Recall from \autoref{Sec-1} that the hyperplane arrangement $\mathcal{B}_n^{[-a,b]}$ in $\mathbb{R}^n$ consists of the following hyperplanes:
	\[
	x_i-x_j=-a,-a+1\ldots,b-1,b,\quad 1\le i<j\le n.
	\]
	This section concerns the counting formula of $r_l(\mathcal{B}_n^{[-a,b]})$. Notice that $t$ is always a factor of the characteristic polynomial $\chi_{\mathcal{B}_n^{[-a,b]}}(t)$. Hence, we can express it as
	\begin{equation}\label{A=tB}
		\chi_{\mathcal{B}_n^{[-a,b]}}(t)=t\tilde{\chi}_{\mathcal{B}_n^{[-a,b]}}(t),
	\end{equation}
	where $\tilde{\chi}_{\mathcal{B}_n^{[-a,b]}}(t):=\frac{1}{t}\chi_{\mathcal{B}_n^{[-a,b]}}(t)$. Given a permutation $w=w_1w_2\ldots w_n\in\mathfrak{S}_n$. For $1\le i\le n-1$, if $w_i>w_{i+1}$, we call $i$ a {\em descent} of $w$. We denote by $d(w)$ the number of descents of $w$. Using the finite field  method, Athanasiadis demonstrated in \cite[(6.6)]{Athanasiadis1996-0} that the polynomial  $\tilde{\chi}_{\mathcal{B}_n^{[0,b]}}(t)$ can be written in the following form
	\begin{equation}\label{CP-0}
		\tilde{\chi}_{\mathcal{B}_n^{[0,b]}}(t)=\sum_{w\in\mathfrak{S}_{n-1}}\binom{t-bd(w)-b-1}{n-1}.
	\end{equation}
	Furthermore, Athanasiadis established a close relationship between characteristic polynomials $\chi_{\mathcal{B}_n^{[-a,b]}}(t)$ and $\chi_{\mathcal{B}_n^{[0,b-a]}}(t)$ in \cite[Theorem 7.1.1]{Athanasiadis1996-0}.
	\begin{theorem}[\cite{Athanasiadis1996-0}, Theorem 7.1.1]\label{CP-Relation}
		Let $a\le b$ be non-negative integers. Then
		\[
		\tilde{\chi}_{\mathcal{B}_n^{[-a,b]}}(t)=\tilde{\chi}_{\mathcal{B}_n^{[0,b-a]}}(t-an).
		\]
	\end{theorem}
	The {\em Eulerian number} $A(n,k)$ is the number of permutations in $\mathfrak{S}_n$ with exactly $k-1$ descents. By applying \eqref{CP-0} to \autoref{CP-Relation}, we can derive an explicit formula for the characteristic polynomial $\chi_{\mathcal{B}_n^{[-a,b]}}(t)$.
	\begin{proposition}\label{Characteristic-Polynomial}
		Let $a\le b$ be non-negative integers. Then
		\[
		\chi_{\mathcal{B}_n^{[-a,b]}}(t)=t\sum\limits_{k=1}^{n-1}A(n-1,k)\binom{t-a(n-k)-bk-1}{n-1}.
		\]
		In particular, when $a=0$,
		\[
		\chi_{\mathcal{B}_n^{[0,b]}}(t) = t\sum\limits_{k=1}^{n-1} A(n-1,k) \binom{t-bk-1}{n-1}.
		\]
	\end{proposition}
	
	For any positive integer $n$, let $(t)_n:=t(t-1)\cdots(t-n+1)$. Then $(t)_n$ can be written in the following forms
	\begin{equation}\label{Stirling-Equation1}
		(t)_n=\sum_{k=0}^{n}s(n,k)t^k=\sum_{k=0}^{n}(-1)^{n-k}c(n,k)t^k,
	\end{equation}
	where a {\em signless Stirling number of the first kind} $c(n,k)$ is the number of permutations in $\mathfrak{S}_n$ with exactly $k$ cycles, and the number
	\[
	s(n,k)=:(-1)^{n-k}c(n,k)
	\]
	is known as a {\em Stirling number of the first kind}. In addition, $t^n$ can be expressed as the form
	\begin{equation}\label{Stirling-Equation2}
		t^n=\sum_{k=0}^{n}S(n,k)(t)_k,
	\end{equation}
	where $S(n,k)$ is the number of partitions of an $n$-set into $k$-blocks  and called a {\em Stirling number of the second kind}.
	
	By comparing the coefficient of $(t)_l$ in different expressions of the characteristic polynomial $\chi_{\mathcal{B}_n^{[-a,b]}}(t)$, we provide a formula for counting regions of level $l$ in $\mathcal{B}_n^{[-a,b]}$.
	\begin{theorem}\label{Number-Level-Region}
		Let $a\le b$ be non-negative integers. For any non-negative integers $n$ and $l$, we have
		\[
		r_l(\mathcal{B}_n^{[-a,b]})=\sum\limits_{k=1}^{n-1}\sum_{j=0}^{n-1}\sum_{i=j}^{n-1}(-1)^{l-1-j}
		\frac{l!\binom{i}{j}A(n-1,k)S(j+1,l)c(n-1,i)\big(a(n-k)+bk+1\big)^{i-j}}{(n-1)!}.
		\]
	\end{theorem}
	\begin{proof}
		The case that $l=0$ or $l>n$ is trivial. We now focus on the case $1\le l\le n$. By \autoref{Characteristic-Polynomial}, we have
		\begin{align}\label{CP-1}
			&\chi(\mathcal{B}_n^{[-a,b]},t)=t\sum\limits_{k=1}^{n-1}\frac{A(n-1,k)}{(n-1)!}\big(t-a(n-k)-bk-1\big)_{n-1}\nonumber\\
			&\stackrel{\eqref{Stirling-Equation1}}=\sum\limits_{k=1}^{n-1}\frac{A(n-1,k)}{(n-1)!}\sum_{i=0}^{n-1}s(n-1,i)t\big(t-a(n-k)-bk-1\big)^i\nonumber\\
			&=\sum\limits_{k=1}^{n-1}\frac{A(n-1,k)}{(n-1)!}\sum_{i=0}^{n-1}s(n-1,i)\sum_{j=0}^i(-1)^{i-j}\binom{i}{j}t^{j+1}\big(a(n-k)+bk+1\big)^{i-j}\nonumber\\
			&=\sum\limits_{k=1}^{n-1}\frac{A(n-1,k)}{(n-1)!}\sum_{j=0}^{n-1}t^{j+1}\sum_{i=j}^{n-1}(-1)^{i-j}\binom{i}{j}s(n-1,i)\big(a(n-k)+bk+1\big)^{i-j}\nonumber\\
			&\stackrel{\eqref{Stirling-Equation2}}=\sum\limits_{k=1}^{n-1}\frac{A(n-1,k)}{(n-1)!}\sum_{j=0}^{n-1}\sum_{l=0}^{j+1}S(j+1,l)(t)_l
			\sum_{i=j}^{n-1}(-1)^{i-j}\binom{i}{j}s(n-1,i)\big(a(n-k)+bk+1\big)^{i-j}\nonumber\\
			&=\sum_{l=0}^n\left[\sum\limits_{k=1}^{n-1}\sum_{j=0}^{n-1}\sum_{i=j}^{n-1}(-1)^{i-j}\binom{i}{j}\frac{A(n-1,k)S(j+1,l)s(n-1,i)\big(a(n-k)+bk+1\big)^{i-j}}{(n-1)!}
			\right](t)_l.
		\end{align}
		From \autoref{Main2}, the characteristic polynomial $\chi_{\mathcal{B}_n^{[-a,b]}}(t)$ can be written in the following form
		\begin{equation}\label{CP-2}
			\chi_{\mathcal{B}_n^{[-a,b]}}(t)=\sum_{l=1}^n(-1)^{n-l}\frac{r_l(\mathcal{B}_n^{[-a,b]})}{l!}(t)_l.
		\end{equation}
		Comparing the coefficient of $(t)_l$ in \eqref{CP-1} and \eqref{CP-2}, we obtain from the relation $s(n-1,i)=(-1)^{n-1-i}c(n-1,i)$ that
		\[
		r_l(\mathcal{B}_n^{[-a,b]})=\sum\limits_{k=1}^{n-1}\sum_{j=0}^{n-1}\sum_{i=j}^{n-1}(-1)^{l-1-j}\frac{l!\binom{i}{j}A(n-1,k)S(j+1,l)c(n-1,i)\big(a(n-k)+bk+1\big)^{i-j}}{(n-1)!}.
		\]
		We complete the proof.
	\end{proof}
	
	We conclude the section with the formulas for counting regions of level $l$ in several typical hyperplane arrangements including (extended) Shi arrangement, (extended) Catalan arrangement and (extended) Linial arrangement.
	\begin{proposition}\label{Extended-Arrangements}
		Let $n$, $l$ and $b\ge 1$ be any non-negative integers. Then
		\begin{itemize}
			\item [{\rm(1)}] $r_l(\mathcal{B}_n^{[-b+2,b]})(t)=\frac{l!}{2^n}\sum_{j=0}^{n-1}\sum_{i=0}^{n}(-1)^{l-1-j}S(j+1,l)\binom{n}{i}\big((b-1)n+i\big)^{n-j-1}$.
			\item [{\rm(2)}]  $r_l(\mathcal{B}_n^{[-b+1,b]})=l\sum_{i=0}^{l-1}(-1)^i\binom{l-1}{i}(bn-i-1)^{n-1}$ {\rm (}see \cite[Theorem 1.3]{Ehrenborg2019}{\rm)}.
			\item [{\rm(3)}]  $r_l(\mathcal{B}_n^{[-b,b]})=\frac{n!bl}{(b+1)n-l}\binom{(b+1)n-l}{bn}$ {\rm (}see \cite[Theorem 6.1, 6.2]{CWYZ2024}\rm{)}.
		\end{itemize}
	\end{proposition}
	\begin{proof}
		The detailed proofs of  (2) and (3) in \autoref{Extended-Arrangements} can be found in  \cite[Theorem 1.3]{Ehrenborg2019} and \cite[Theorem 6.1, 6.2]{CWYZ2024}, respectively. So, here we only give a proof of (1) in \autoref{Extended-Arrangements}. From \cite[(9.11) in Example 9.10]{P-S2000},
		we have that the characteristic polynomial of the extended Linial arrangement $\mathcal{B}_n^{[-b+2,b]}$ is
		\[
		\chi_{\mathcal{B}_n^{[-b+2,b]}}(t)=\frac{t}{2^n}\sum_{i=0}^{n}\binom{n}{i}\big(t-(b-1)n-i\big)^{n-1}.
		\]
		Similar to \eqref{CP-1}, $\chi_{\mathcal{B}_n^{[-b+2,b]}}(t)$ can be written in the following form
		\begin{equation}\label{CP-3}
			\chi_{\mathcal{B}_n^{[-b+2,b]}}(t)=\sum_{l=0}^{n}\Big[\frac{1}{2^n}\sum_{j=0}^{n-1}\sum_{i=0}^{n}(-1)^{n-j-1}S(j+1,l)\binom{n}{i}
			\big((b-1)n+i\big)^{n-j-1}\Big](t)_l.
		\end{equation}
		Likewise, comparing the coefficient of  $(t)_l$ in \eqref{CP-2} (taking $a=b-2$) and \eqref{CP-3}, we obtain
		\[
		r_l(\mathcal{B}_n^{[-b+2,b]})=\frac{l!}{2^n}\sum_{j=0}^{n-1}\sum_{i=0}^{n}(-1)^{l-1-j}S(j+1,l)\binom{n}{i}\big((b-1)n+i\big)^{n-j-1}.
		\]
		This completes the proof.
	\end{proof}
	\section{Real roots of $\chi_{\mathcal{B}_n^{[-a,b]}}(t)$}\label{Sec-5}
	Postnikov and  Stanley in \cite[Theorem 9.12]{P-S2000} showed that all roots of the polynomial $\tilde{\chi}_{\mathcal{B}_n^{[-a,b]}}(t)$ of the hyperplane arrangement $\mathcal{B}_n^{[-a,b]}$ with $a\ne b$ have the same real part equal to $\binom{n(a+b+1)}{2}$. Motivated by their work, here we focus on exploring the real roots of  $\chi_{\mathcal{B}_n^{[-a,b]}}(t)$.
	\begin{theorem}\label{Main4}
		Let $n\ge 2$ and $b$ be integers with $b\ge n-1$. Then the real roots of the characteristic polynomial $\chi_{\mathcal{B}_n^{[0,b]}}(t)$ {\rm(}$\chi_{\mathcal{B}_n^{[-b,0]}}(t)$, resp.{\rm)} are determined by the following cases:
		\begin{itemize}
			\item [{\rm(1)}]  If $n$ is odd, then $\chi_{\mathcal{B}_n^{[0,b]}}(t)$ {\rm(}$\chi_{\mathcal{B}_n^{[-b,0]}}(t)$, resp.{\rm)} has a single real root at $0$  of multiplicity one.
			\item [{\rm(2)}]  If $n$ is even, then $\chi_{\mathcal{B}_n^{[0,b]}}(t)$ {\rm(}$\chi_{\mathcal{B}_n^{[-b,0]}}(t)$, resp.{\rm)} has exactly two real roots at $0$ and $\frac{n(b+1)}{2}$, each with  multiplicity one.
		\end{itemize}
	\end{theorem}
	\begin{proof}
		Since $\chi_{\mathcal{B}_n^{[0,b]}}(t)=\chi_{\mathcal{B}_n^{[-b,0]}}(t)$, we only need to consider the real roots of the characteristic polynomial $\chi_{\mathcal{B}_n^{[0,b]}}(t)$.
		From \autoref{Characteristic-Polynomial}, we have
		\[
		\chi_{\mathcal{B}_n^{[0,b]}}(t) =\frac{t}{(n-1)!}\sum\limits_{k=1}^{n-1} A(n-1,k)(t-bk-1)_{n-1}.
		\]
		Let $A_k(t):=A(n-1,k)(t-bk-1)_{n-1}$ for $k=1,\ldots,n-1$. Then $\tilde{\chi}_{\mathcal{B}_n^{[0,b]}}(t)=\frac{1}{(n-1)!}\sum\limits_{k=1}^{n-1} A_k(t)$.
		Clearly, $0$ is a real root of the characteristic polynomial $\chi_{\mathcal{B}_n^{[0,b]}}(t)$ through the relation $\chi_{\mathcal{B}_n^{[0,b]}}(t)=t\tilde{\chi}_{\mathcal{B}_n^{[0,b]}}(t)$ in \eqref{A=tB}. We now examine the real roots of the polynomial $\tilde{\chi}_{\mathcal{B}_n^{[0,b]}}(t)$.
		
		When $n$ is odd, we derive from $A(n-1,k)=A(n-1,n-k)$ that
		\begin{equation*}\label{Symmetry1}
			\tilde{\chi}_{\mathcal{B}_n^{[0,b]}}(t)=\tilde{\chi}_{\mathcal{B}_n^{[0,b]}}\big(n(b+1)-t\big).
		\end{equation*}
		This implies that the polynomial $\tilde{\chi}_{\mathcal{B}_n^{[0,b]}}(t)$  is symmetric about the line $t=\frac{n(b+1)}{2}$. Moreover, the real numbers $bk+i$ and $b(n-k)+n-i$ are symmetric about $\frac{n(b+1)}{2}$ for $k,i=1,\ldots,\frac{n-1}{2}$ as well. This means that for any $t< \frac{n(b+1)}{2}$, we have
		\begin{equation*}
			|t-bk-i|<|t-b(n-k)-(n-i)|=b(n-k)+(n-i)-t,\quad k,i=1,\ldots,\frac{n-1}{2}.
		\end{equation*}
		By $A(n-1,k)=A(n-1,n-k)$, this further leads to that for any $t\le \frac{n(b+1)}{2}$,
		\[
		A_k(t)+A_{n-k}(t)>0, \quad k=1,\ldots,\frac{n-1}{2}.
		\]
		By symmetry, we deduce $A_k(t)+A_{n-k}(t)>0$ for all $t\ge \frac{n(b+1)}{2}$. Therefore, we have $\tilde{\chi}_{\mathcal{B}_n^{[0,b]}}(t)>0$ for all real numbers. Then, by \eqref{A=tB} again, we have shown that $\chi_{\mathcal{B}_n^{[0,b]}}(t)$ has a unique real root at $0$ in this case.
		
		When $n$ is even, taking $t=\frac{n(b+1)}{2}$ and applying $A(n-1,k)=A(n-1,n-k)$ again, we obtain
		\[
		A_k\Big(\frac{n(b+1)}{2}\Big)+A_{n-k}\Big(\frac{n(b+1)}{2}\Big)=0\For k=1,\ldots\frac{n}{2}-1\quad\And\quad A_{\frac{n}{2}}\Big(\frac{n(b+1)}{2}\Big)=0.
		\]
		So, $\tilde{\chi}_{\mathcal{B}_n^{[0,b]}}\big(\frac{n(b+1)}{2}\big)=0$. This means that $\tilde{\chi}_{\mathcal{B}_n^{[0,b]}}(t)$ has a real root at $\frac{n(b+1)}{2}$ in the case. It is clear that the derivative of $\tilde{\chi}_{\mathcal{B}_n^{[0,b]}}(t)$ is
		\[
		\tilde{\chi}'_{\mathcal{B}_n^{[0,b]}}(t)=\frac{1}{(n-1)!}\sum_{k=1}^{n-1}\sum_{i=1}^{n-1}\frac{A_k(t)}{t-bk-i}.
		\]
		By routine calculations, we can also obtain
		\[
		\tilde{\chi}'_{\mathcal{B}_n^{[0,b]}}(t)=\tilde{\chi}'_{\mathcal{B}_n^{[0,b]}}\big(n(b+1)-t\big).
		\]
		Namely, the polynomial $\tilde{\chi}'_{\mathcal{B}_n^{[0,b]}}(t)$ is symmetric about the line $t=\frac{n(b+1)}{2}$. Similar to the case that $n$ is odd, we can verify $\tilde{\chi}'_{\mathcal{B}_n^{[0,b]}}(t)>0$ for any $t\in\mathbb{R}$. Therefore, $\tilde{\chi}_{\mathcal{B}_n^{[0,b]}}(t)$ is a strictly increasing function on $\mathbb{R}$. Then we have $\tilde{\chi}_{\mathcal{B}_n^{[0,b]}}(t)<\tilde{\chi}_{\mathcal{B}_n^{[0,b]}}\big(\frac{n(b+1)}{2}\big)=0$ if $t<\frac{n(b+1)}{2}$, and $\tilde{\chi}_{\mathcal{B}_n^{[0,b]}}(t)>\tilde{\chi}_{\mathcal{B}_n^{[0,b]}}\big(\frac{n(b+1)}{2}\big)=0$ if $t>\frac{n(b+1)}{2}$. This implies that $\tilde{\chi}_{\mathcal{B}_n^{[0,b]}}(t)$ has a unique real root at $\frac{n(b+1)}{2}$ in the case. It follows from \eqref{A=tB} that $\chi_{\mathcal{B}_n^{[0,b]}}(t)$ has exactly two simple real roots: $0$ and $\frac{n(b+1)}{2}$ in the case. We finish the proof.
	\end{proof}
	
	Notably, Athanasiadis showed in \cite[Theorem 6.4.3]{Athanasiadis1996-0} or \cite[Theorem 4.3]{Athanasiadis1996} that the characteristic polynomials $\chi_{\mathcal{B}_n^{[0,b]}}(t)$ and $\chi_{\mathcal{B}_n^{[b]}}(t)$ are related by
	\begin{equation}\label{CP-Relation1}
		\tilde{\chi}_{\mathcal{B}_n^{[b]}}(t)=\tilde{\chi}_{\mathcal{B}_n^{[0,b+1]}}(t+n).
	\end{equation}
	As a direct consequence of \autoref{Main4} and \eqref{CP-Relation1}, we arrive at the following corollary.
	\begin{corollary}\label{Main-4-1}
		Let $n\ge 2$ and $b$ be positive integers with $b\ge n-2$. Then the real roots of the characteristic polynomial $\chi_{\mathcal{B}_n^{[b]}}(t)$ {\rm(}$\chi_{\mathcal{B}_n^{[-b]}}(t)$, resp.{\rm)} are determined by the following cases:
		\begin{itemize}
			\item [{\rm(1)}]  If $n$ is odd, then $\chi_{\mathcal{B}_n^{[b]}}(t)$ {\rm(}$\chi_{\mathcal{B}_n^{[-b]}}(t)$, resp.{\rm)} has a single real root at $0$ of multiplicity one.
			\item [{\rm(2)}]  If $n$ is even, then $\chi_{\mathcal{B}_n^{[b]}}(t)$ {\rm(}$\chi_{\mathcal{B}_n^{[-b]}}(t)$, resp.{\rm)} has exactly two real roots at $0$ and $\frac{nb}{2}$, each with  multiplicity one.
		\end{itemize}
	\end{corollary}
	\begin{proof}
		As $\chi_{\mathcal{B}_n^{[b]}}(t)=\chi_{\mathcal{B}_n^{[-b]}}(t)$, we only need to consider the real roots of the characteristic polynomial $\chi_{\mathcal{B}_n^{[b]}}(t)$. It follows from \autoref{Main4} and \eqref{CP-Relation1} that if $n$ is odd, $\chi_{\mathcal{B}_n^{[b]}}(t)$ has a unique real root at $0$ and if $n$ is even, $\chi_{\mathcal{B}_n^{[b]}}(t)$ has two simple real roots : $0$ and $\frac{nb}{2}$. We complete the proof.
	\end{proof}
	
	More generally, we can directly get the real roots of the characteristic polynomial $\chi_{\mathcal{B}_n^{[-a,b]}}(t)$ of $\mathcal{B}_n^{[-a,b]}$ via \autoref{CP-Relation} and \autoref{Main4}.
	\begin{corollary}\label{Main-4-2}
		Let $n\ge 2$, $a$ and $b$ be non-negative integers with $b-a\ge n-1$. Then the real roots of the characteristic polynomial $\chi_{\mathcal{B}_n^{[-a,b]}}(t)$ are determined by the following cases:
		\begin{itemize}
			\item [{\rm(1)}]  If $n$ is odd, then $\chi_{\mathcal{B}_n^{[-a,b]}}(t)$ has a single real root at $0$ of multiplicity one.
			\item [{\rm(2)}]  If $n$ is even, then $\chi_{\mathcal{B}_n^{[-a,b]}}(t)$ has exactly two real roots at $0$ and $\frac{n(a+b+1)}{2}$, each with  multiplicity one.
		\end{itemize}
	\end{corollary}
	\section*{Acknowledgements}
	The second author is supported by National Natural Science Foundation of China under Grant No. 12301424.
	
\end{document}